\documentclass[preprint,12pt]{elsarticle}
\journal{Advances in Applied Mathematics}
\usepackage[utf8]{inputenc}
\usepackage{amssymb}
\usepackage{amsmath}
\usepackage{amsthm}
\usepackage[mathscr]{eucal}
\usepackage{graphicx}
\usepackage[hidelinks]{hyperref}
\makeatletter

\def\Fraisse{Fra\"{\i}ss\' e}

\def\B{{\cal B}}
\def\C{{\cal C}}
\def\F{{\cal F}}
\def\Forb{\mathop\mathrm {Forb}\nolimits}
\def\Age{\mathop{\mathrm{Age}}\nolimits}
\def\ACl{\mathop{\mathrm{Acl}}\nolimits}
\def\Rel{\mathrm {Rel}}

\def\arity#1{a(\rel{}{#1})}
\def\sh{\mathop{\mathrm{Sh}}\nolimits}

\def\str#1{\mathbf {#1}}

\def\rel#1#2{R_{\mathbf{#1}}^{#2}}
\def\nbrel#1#2{R_{#1}^{#2}}

\theoremstyle{definition}
\newtheorem*{example}{Example}
\newtheorem*{remark}{Remark}
\theoremstyle{remark}

\theoremstyle{plain}
\newtheorem{thm}{Theorem}[section]
\newtheorem{corollary}[thm]{Corollary}

\newtheorem{lem}[thm]{Lemma} 

\theoremstyle{definition}
\newtheorem{defn}[thm]{Definition}

\begin{document}
\bibliographystyle{plain}

\begin{frontmatter}
\title{Bowtie-free graphs have a Ramsey lift}

\author[KAM]{Jan Hubi\v cka\fnref{g1}}
\ead{hubicka@kam.mff.cuni.cz}
\fntext[g1]{Supported by grant ERC-CZ LL-1201 of the Czech Ministry of Education and CE-ITI P202/12/G061 of GA\v CR. This research was partially done while the authors took part in Trimester Universality and Homogeneity at Hausdorff Institute (Bonn) in the fall 2013.}
\author[IUUK]{Jaroslav Ne\v set\v ril\fnref{g1}}
\ead{nesetril@iuuk.mff.cuni.cz}
\address[KAM]{Departement of Applied Mathematics (KAM)\\ Charles University\\ Prague, Czech Republic}
\address[IUUK]{Computer Science Institute of Charles University (IUUK)\\ Charles University\\ Prague, Czech Republic}

\begin{abstract}
A bowtie is a graph consisting of two triangles with one vertex identified.  We
show that the class of all (finite) graphs not containing a bowtie as a
subgraph has a Ramsey lift  (expansion). This solves one of the old problems
in the area and it is the first Ramsey class with a non-trivial
algebraic closure.
\end{abstract}
\end{frontmatter}

\section{Introduction}
\begin{figure}[h]
\centerline{\includegraphics{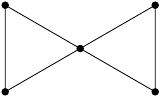}}
\caption{The bowtie graph.}
\label{fig:bowtie}
\end{figure}

A {\em bowtie} graph\footnote{This poetic name seems to be first used in \cite{Ringeisen1985}, see also \cite{Komjath1999}, 
{\em butterfly graph} or {\em hourglass graph} are other names used; it is ``\textbackslash bowtie'' in TeX and sign for ``natural join'' in databases.
Interestingly, bowtie appears in~\cite{Graham1990}}
 is formed by two triangles intersecting in a single vertex (see Figure~\ref{fig:bowtie}).
We denote by $\B$ the class of all finite graphs not containing a bowtie as a
(not necessarily induced) subgraph.
The class $\B$ seems to be a rather special class.
However, it appears that it plays a key role in the context of both Ramsey theory and model theory in the area related to
universality and homogeneity.
 It is the interplay of these two fields which makes this example interesting and important. We briefly explain both sides and their interplay in this introduction and in Section~\ref{remarks}.

\subsection{Ramsey Theory}
Ramsey theory (see~\cite{Graham1990,Nevsetvril1995} for background information) is established in the context of several mathematical areas. Structural Ramsey theory is
interested in generalisations of the Ramsey Theorem to as wide class of structures as possible. The key notion in this area is the Ramsey class.
To make this paper self-contained, we introduce it in the following notation (which is by now standard, see e.g. \cite{Nevsetvril1995}).

Let $\C$ be a class of structures endowed with embeddings. The class is usually
understood from the context. Let $\str{A},\str{B}$ be objects of $\C$.
Then by ${\str{B}\choose \str{A}}$  we denote the set of all sub-objects
$\widetilde{\str{A}}$ of $\str{B}$, $\widetilde{\str{A}}$ isomorphic to $\str{A}$. (By a
sub-object we mean that the inclusion is an embedding.) Using this notation the
definition of Ramsey class gets the following form:
A class $\C$ is a {\em Ramsey class} if for every two objects $\str{A},\str{B}\in \mathcal C$ and for every positive integer $k$ there exists object $\str{C}\in \mathcal C$ such that 
 for every partition of ${\str{C}\choose \str{A}}$ in $k$ classes there exists $\widetilde{\str B} \in {\str{C}\choose \str{B}}$ such that ${\widetilde{\str{B}}\choose \str{A}}$ belongs to one class of the partition.
It is usual to shorten the last part of the definition as $\str{C} \longrightarrow (\str{B})^{\str{A}}_2$.

The Ramsey classes originated in 70's (see \cite{Nevsetvril1995}) as the top of the line of Ramsey properties and examples found present the backbone of the structural Ramsey theory, see \cite{Nevsetvril1976a,Nevsetvril1989a,Nevsetvril1995}.

In most instances, a class is not Ramsey for some easily formulated reason and all one needs is to add some more information such as ordering or colouring of distinguished parts. For example, all finite graphs form a Ramsey class if we add an ordering of vertices, bipartite graphs need an ordering respecting the bipartition and colouring distinguishing the parts, disjoint unions of complete
graphs (or equivalences) needs an ordering respecting components. This additional information is usually called an expansion, or in a combinatorial setting a lift, of the original structure (see the next section). 

Bowtie-free graphs do not form Ramsey class.  We need not only linear ordering of vertices but also more complicated lifts
with seven types of vertices and large amount of types of edges, see Sections~\ref{sec:bowtie} and~\ref{sec:homogenization}.
As a simple example consider bowtie-free graph depicted in Figure~\ref{fig:bowtie-counterexample}. This graph contains two types of edges: edges in precisely one triangle and edges in two.  It can be verified by hand that this graph can not be extended to a bowtie-free graph where every edge is in multiple triangles. Consequently edges of every copy of this graph within a bowtie-free graph can be coloured red if they are in precisely one triangle and blue otherwise which contradicts the edge Ramsey property.
\begin{figure}
\centerline{\includegraphics{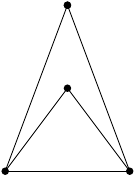}}
\caption{A counter-example for edge Ramsey property for bowtie-free graphs.}
\label{fig:bowtie-counterexample}
\end{figure}

\subsection{Model theory}
It is important to realise that for more complicated Ramsey questions (even when related to graphs) one needs to deal with more general structures.

A language $L$ is a set of relational symbols $\rel{}{}\in L$, each associated with natural number $\arity{}$ called \emph{arity}.
A \emph{(relational) $L$-structure} $\str{A}$ is a pair $(A,(\rel{A}{};\rel{}{}\in L))$ where $\rel{A}{}\subseteq A^{\arity{}}$ (i.e. $\rel{A}{}$ is a $\arity{}$-ary relation on $A$). The set $A$ is called the \emph{vertex set} or the \emph{domain} of $\str{A}$ and elements of $A$ are \emph{vertices}. 
The language is usually fixed and understood from the context (and it is in most cases denoted by $L$).
However it is the essence of this paper that  the languages considered are complex and we consider an interplay of several of them. This will be carefully described. 
If set $A$ is finite we call \emph{$\str A$ finite structure}. We consider only structures with finitely or countably many vertices.
The class of all (finite or countable) relational $L$-structures will be denoted by $\Rel(L)$.

We consider graphs as a special case of relational structure with one binary
relation.  We use bold letters $\str{A}$, $\str{B}$, \ldots{} to denote
structures and normal letters $G$, $H$, \ldots{} for graphs. The following are
standard graph theoretic notions re-stated in the language of model theory.
A \emph{homomorphism} $f:\str{A}\to \str{B}=(B,(\rel{B}{};\rel{}{}\in L))$ is a mapping $f:A\to B$ satisfying for every $\rel{}{}\in L$ the implication $(x_1,x_2,\ldots, x_{\arity{}})\in \rel{A}{}\implies (f(x_1),f(x_2),\ldots,f(x_{\arity{}}))\in \rel{B}{}$. (For a subset $A'\subseteq A$ we denote by $f(A')$ the set $\{f(x);x\in A'\}$ and by $f(\str{A})$ the homomorphic image of a structure.) 
If $f$ is injective, then $f$ is called a \emph{monomorphism}. A monomorphism is called \emph{embedding} if the above implication is equivalence, i.e. if for every $\rel{}{}\in L$ we have $(x_1,x_2,\ldots, x_{\arity{}})\in \rel{A}{}\iff (f(x_1),f(x_2),\ldots,f(x_{\arity{}}))\in \rel{B}{}$.  If $f$ is an embedding which is an inclusion then $\str{A}$ is a \emph{substructure} (or \emph{subobject}) of $\str{B}$. Note that substructures correspond to induced subgraphs. For an embedding $f:\str{A}\to \str{B}$ we say that $\str{A}$ is \emph{isomorphic} to $f(\str{A})$ and $f(\str{A})$ is also called a \emph{copy} of $\str{A}$ in $\str{B}$. Thus $\str{B}\choose \str{A}$ is defined as the set of all copies of $\str{A}$ in $\str{B}$. All copies considered in this paper are thus induced.

Using the language of model theory we can conveniently define the concept of lift discussed informally in the previous section.
Let $L^+$ be a language containing language $L$. By this we mean $L\subseteq L^+$ and the arities of the relations both in $L$ and $L^+$ are the same.
Then every structure $\str{X}=(X,(\rel{X}{}; \rel{}{}\in L^+))\in \Rel(L^+)$ may be viewed as a structure $\str{A}=(X,(\rel{X}{}; \rel{}{}\in L))\in \Rel(L)$ together with some additional relations $\rel{X}{}$ for $\rel{}{}\in L^+\setminus L$.
We call $\str{X}$ a {\em lift}. (In the model theory context lift is usually called an {\em expansion}.)
 In this situation  the structure $\str{A}$ is called
the {\em shadow} (or  alternatively  the {\em reduct}) of $\str{X}$. The class $\Rel(L^+)$ is
the class of all lifts of $\Rel(L)$.  
Conversely, $\Rel(L)$ is the
class of all shadows of $\Rel(L^+)$. In this paper the languages $L$ and $L^+$ will always be finite, we speak about {\em finite lifts}.
Given class of relational $L$-structures $\mathcal K$, we call class $\mathcal K^+$ of $L^+$-structures a {\em lift} of $\mathcal K$ if for every $\str{A}\in \mathcal K$ there is $\str{A}^+\in \mathcal K^+$ which is a lift of $\str{A}$ and moreover every shadow of a structure in $\mathcal K^+$ is in $\mathcal K$.

Two notions are related to our main result: A (countable) structure $\str{A}$ is said to be {\em universal for a class $\mathcal C$} of (finite or countably infinite) structures if every structure $\str{B}\in \mathcal C$ embeds to $\str{A}$.
A relational structure $\mathbf{A}$ is called {\em ultrahomogeneous} if every isomorphism between two induced finite
substructures of $\str{A}$ can be extended to an automorphism of
$\str{A}$.

It is a classical result of model theory that ultrahomogeneous structures may be alternatively described as \Fraisse{} limits of amalgamation classes of finite structures (see e.g. \cite{Hodges1993}). Here {\em amalgamation class} $\mathcal C$ is a hereditary class of structures containing only countably many mutually non-isomorphic structures which satisfy:
\begin{enumerate}
\item ({\em Joint embedding property}) For every $\str{A}, \str{B}\in \C$ there exists $\str{C}\in \C$ such that $\str{C}$ contains both $\str{A}$ and $\str{B}$ as substructures
\item ({\em Amalgamation property})
For $\str{A},\str{B}_1,\str{B}_2\in \C$ and $\alpha_1$ embedding of $\str{A}$ into $\str{B}_1$, $\alpha_2$ embedding of $\str{A}$ into $\str{B}_2$, there is $\str{C}\in \C$
 with embeddings $\beta_1:\str{B}_1 \to \str{C}$ and
$\beta_2:\str{B}_2\to\str{C}$ such that $\beta_1\circ\alpha_1 =
\beta_2\circ\alpha_2$. Every such structure $\str{C}$ is called an \emph{amalgamation} of $\str{B}_1$ and $\str{B}_2$ over $\str{A}$ with respect to $\alpha_1$ and $\alpha_2$. 
\end{enumerate}
We say that an amalgamation is \emph{strong} when $\beta_1(x_1)=\beta_2(x_2)$ if and
only if $x_1\in \alpha_1(A)$ and $x_2\in \alpha_2(A)$.  Less formally, a strong
amalgamation glues together $\str{B}_1$ and $\str{B}_2$ with an overlap no
greater than the copy of $\str{A}$ itself.  A strong amalgamation is \emph{free} if there are no tuples in any relations of $\str{C}$ containing both vertices of
$\beta_1(B_1\setminus \alpha_1(A))$ and $\beta_2(B_2\setminus \alpha_2(A))$.

For a structure $\str{A}$ the {\em age of $\str{A}$}, denoted by $\Age(\str{A})$, is the class of all finite structures which have embedding to $\str{A}$. Thus every homogeneous structure $\str{A}$ is determined by $\Age(\str{A})$ which forms an amalgamation class (see \cite{Hodges1993}).

Let $\str{A}$ be an $L$-relational structure and $S$ a finite subset of $A$.
The \emph{algebraic closure of $S$ in $\str{A}$}, denoted by $\ACl_\str{A}(S)$,
is the set all vertices $v\in A$ for which there is a formula $\phi$ in the
language $L$ with $\vert S\vert +1$ variables such that $\phi(\vec{S},v)$ is
true and there are only finitely many vertices $v'\in A$ such that
$\phi(\vec{S},v')$ is also true. (Here $\vec{S}$ is an arbitrary ordering of
vertices of $S$.)

Algebraic closure is, of course, related to amalgamation: For example, it is easy to see that if an ultrahomogeneous structure $\str{H}$ has
trivial closure (i.e. $\ACl_\str{H}(S)=S$ for every $S\subseteq H$) then its age is closed for strong amalgamation~\cite{Cherlin1999}.

The link between Ramsey classes and ultrahomogeneous structures was established in~\cite{Nevsetvril1989a}:
 Under a mild assumption any Ramsey class is an amalgamation class and thus it is an age of 
an ultrahomogeneous structure. This was used in \cite{Nevsetvril1989a} to completely characterise hereditary Ramsey classes   of undirected graphs. 
(Essentially, all Ramsey classes of graphs were known earlier, \cite{Nevsetvril1977b}.)
This connection of Ramsey classes proved to be fruitful and led to the characterisation programme for Ramsey classes \cite{Nevsetril2005}  and to an important connection of Ramsey classes with topological dynamics and ergodic theory \cite{Kechris2005}.

As we indicated above a given class $\mathcal C$ is often not Ramsey but $\mathcal C$  may have an easy lift $\mathcal C^+$ which is Ramsey and thus it leads to the age of an ultrahomogeneous structure~$\str{U}^+$. This in turn means that the shadow $\str{U}$ of $\str{U}^+$ is universal for $\mathcal C$. In this sense \emph{the universality is the first test for the existence of a Ramsey lift.}

We briefly comment on this connection at the end of this paper in Section~\ref{remarks}.

\subsection{Statement of results}

The rest of the paper is organised as follows:  
Section~\ref{sec:bowtie} contains a detailed description of the structure  of bowtie-free graphs. This leads to an explicit homogenisation of these graphs by means of lifts $L_0$, $L_1$ and $L_2$ which will be introduced in Section~\ref{sec:homogenization}. In Section~\ref{sec:reduced} we give a simpler (reduced) variant of our lift. In Section~\ref{sec:ramsey} we review the basic properties of Ramsey classes used in our proof.
The proof of Ramsey property splits into two parts: In Section~\ref{sec:star} we prove the Ramsey property for incomplete lifts,
 and finally in Section~\ref{sec:final} we combine this to obtain our main result:
\begin{thm}
\label{thm:intro}
The class $\B$ of all finite bowtie-free graphs has a finite Ramsey lift.
\end{thm}

In a more detailed way this is formulated as Theorem~\ref{thm:mainIII} below.
 This Ramsey lift includes special ``admissible'' orderings. As an explanation of this and as an application
of Theorem~\ref{thm:intro} we then prove the lift property for the class of orderings in Section~\ref{sec:expansion} (lift property is introduced there).
The final section contains some remarks and open problems and comments of the relationship of Ramsey classes and universal structures.

\section{Structure of bowtie-free graphs}
\label{sec:bowtie}

In order to prove the Ramsey property one has to understand the lifted class very well and the lift has to be  explicit.
We start to develop the structure of bowtie-free graphs by means of the following concepts which will describe the structure of triangles in bowtie-free graphs. 

\begin{defn} [Chimneys]
For $n\geq 2$, an {\em $n$-chimney graph}, $Ch_n$, is a free amalgamation of $n$ triangles over one common edge.
A {\em chimney graph} is any graph isomorphic to $Ch_n$ for some $n\geq 2$.
\end{defn}

Chimneys together with $K_4$ (a clique on 4 vertices) will form the only components of bowtie-free graphs formed by triangles. The assumption $n \geq 2$  for chimney is a technical assumption to avoid isolated triangles. Note also that $Ch_2$ is not an induced subgraph of $K_4$.
 
\begin{defn}[Good bowtie-free graphs]
\label{def:goodgraph}
A bowtie-free graph $G=(V,E)$ is {\em good} if every vertex is contained either in an induced copy of chimney graph or a copy of the complete graph $K_4$.
\end{defn}

The structure of bowtie-free graphs is captured by means of the following three lemmas:

\begin{lem}
\label{lem:bowtiestructure}
Every bowtie-free graph $G$ is an induced subgraph of some good bowtie-free graph $G'$.
\end{lem}
\begin{proof}
Every bowtie-free graph $G$ can be extended in the following way:
\begin{enumerate}
 \item[1.] For every vertex $v$ not contained in a triangle add a new induced copy of $Ch_2$ and identify vertex $v$ with one of vertices of $Ch_2$.
 \item[2.] For every triangle $v_1,v_2,v_3$ that is not part of a 2-chimney nor $K_4$ add a new vertex $v_4$ and triangle $v_1,v_2,v_4$
       turning the original triangle into $Ch_2$.
\end{enumerate}
It is easy to see that step $1.$ can not introduce new bowtie.

Assume, to the contrary, that step $2.$ introduced a new bowtie. Further assume that $v_1$ is the
unique vertex of degree 4 of this new bowtie and consequently there is another
triangle on vertex $v_1$ in $G$.  Because $G$ is bowtie-free, this triangle
must share a common edge with triangle $v_1,v_2,v_3$ and therefore
triangle $v_1,v_2,v_3$ is already part of $K_4$ or a 2-chimney in the original graph $G$.  A
contradiction.
\end{proof}

For a bowtie-free graph $G=(V,E)$ we split its edge set into two types: $E_0 = E_0(G)$ consisting of all
edges in triangles and $E_1 = E_1(G)$ consisting of all remaining edges. We also speak about {\em edges of type $0$}  and {\em edges of type $1$}. 
Put also $G_0=(V,E_0)$.

\begin{lem}
\label{lem:chimneys}
\label{lem:bowtiefree}
For every good bowtie-free graph $G=(V,E)$ the graph $G_0$ is a disjoint union of
induced copies of chimneys and $K_4$.

Conversely let $G=(V,E)$ be a graph with every vertex contained either in an induced copy of chimney $Ch_n$, $n\geq 2$, or a copy of the complete graph $K_4$. If graph $G_0=(V,E_0)$ is a disjoint union of induced copies of chimneys and $K_4$ and the remaining edges of $G$ (i.e. edges in $E_1$) do not form a triangle, then $G$ is a good bowtie-free graph.
\end{lem}

\begin{proof}
First part of the statement follows directly from the fact that $Ch_n$, $n\geq 1$, and $K_4$ are
the only connected bowtie-free graphs with every vertex and edge in a triangle.

The second part of the statement follows from the fact that a bowtie in $G$ must be a bowtie in $G_0$ and $G_0$ is a bowtie-free by the assumption.
\end{proof}

\begin{figure}
\centerline{\includegraphics{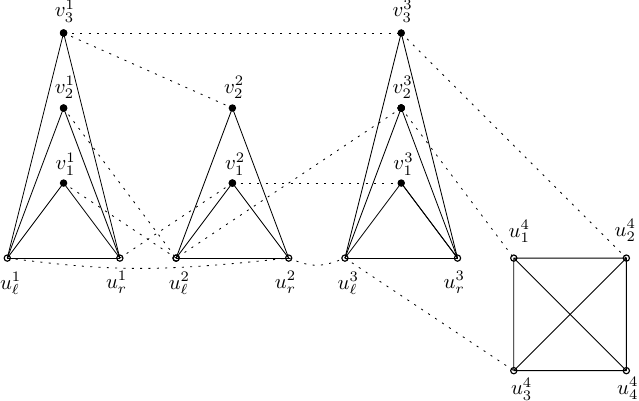}}
\caption{An example of a good  bowtie-free graph.}
\label{fig:bowtie-example}
\end{figure}

It follows that bowtie-free graphs can be extended to good bowtie-free graphs that are made of chimneys and $K_4$'s (forming the edge set $E_0$) and a triangle free graph (with the edge set $E_1$). 
An example of a good bowtie-free graph is depicted in Figure~\ref{fig:bowtie-example}. Type 0 edges are depicted as solid lines,
type 1 edges are dashed. We will use this graph as our reference graph through the paper.

\section{(Ultra)Homogenisation of $\B$}
\label{sec:homogenization}
{\em Homogenisation} is a technique which provides an  ultrahomogeneous lift  for a  non-ultrahomogeneous structure.
The special structure of good  bowtie-free graphs indicates that we have vertices of various types and that the Ramsey lift will 
have to be defined carefully (to distinguish all possible combination of types). In this section we shall define three  lifts (with languages $L_0$, $L_1$ and 
$L_2$)  and use them to define an amalgamation class (see Corollary~\ref{cor:ultrahomogeneous}).
We start with the definition of centre.

 Let $G$ be a good bowtie-free graph.  Then the {\em centre} of $G$, $c(G)$, is a subgraph
 induced by all vertices contained in two or more triangles.
The {\em centre of a vertex} $v$, denoted by $c(v)$, is a subgraph of $G$ induced by all vertices
in two or more triangles which are in the same connectivity component of $(V, E_0)$ as the vertex $v$. (We define centre for good bowtie-free graphs only, so every vertex has a centre.)  According to Lemma \ref{lem:chimneys}, the centre of a vertex is either an edge (if $v$ is
contained in a chimney) or $K_4$ (if $v$ is contained in a copy of $K_4$).
We also call a vertex {\em central} if it appears in centre. Other vertices are {\em non-central}.

\begin{remark}
Note that in the language of model theory the centre of a vertex is a definable set and thus bowtie-free graphs have a nontrivial algebraic closure (and as explained above this was one of the motivations for a study of this particular example, see e.g. \cite{Cherlin2011}).
\end{remark}

\begin{example}
Our reference graph depicted in Figure~\ref{fig:bowtie-example} has central vertices labelled $u$ and non-central $v$.
There are 4 centres of vertices: $\{u^1_\ell, u^1_r\}$,  $\{u^2_\ell, u^2_r\}$,  $\{u^3_\ell, u^3_r\}$, and $\{u^4_1,u^4_2,u^4_3,u^4_4\}$.
The centre of vertex $v^1_1$ is $\{u^1_\ell, u^1_r\}$.  The centre of $u^4_1$ is $\{u^4_1,u^4_2,u^4_3,u^4_4\}$.
\end{example}

We start with the following (easy and optimistic) statement:

\begin{lem}[Central amalgamation]
\label{lem:amalgamation}
Let $G$ and $G'$ be good  bowtie-free graphs and $f$ an isomorphism from $c(G)$ to $c(G')$.
Then the free amalgamation of $G$ and $G'$ over central vertices (with respect to $f$) is a good bowtie-free 
graph.
\end{lem}
\begin{figure}
\centerline{\includegraphics{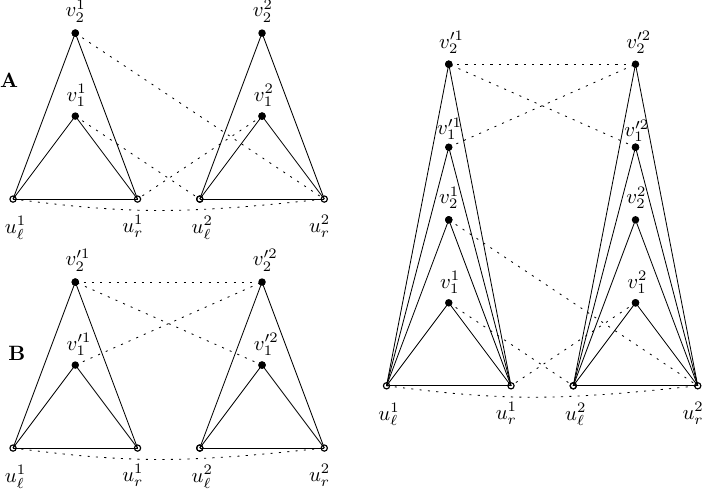}}
\caption{Structures $\str{A}$ and $\str{B}$ and their amalgamation over the common centre.}
\label{fig:amalgamation}
\end{figure}

\begin{proof}
Without loss of generality we can assume that $f$ is an identity and vertex
sets of $G=(V,E)$ and $G'=(V',E')$ intersect only on vertices of $c(G)$.  The free
amalgamation is a graph $G''=(V\cup V',E\cup E')$.

All triangles of $G''$ are clearly either triangles in $G$ or $G'$ (or both).
$G''$ is good because all copies of $K_4$ are also part of centres and thus
identified.  For vertices contained in chimneys, some vertices of chimney $Ch_n$ of $G$ gets
identified with some vertices of chimney $Ch_m$ of $G'$ if an only if centres of the chimneys are the 
same.  This produces a chimney $Ch_{n+m}$ in $G''$.
$G''$ is bowtie-free by Lemma \ref{lem:bowtiefree}.
\end{proof}

An example of the central amalgamation is depicted in Figure~\ref{fig:amalgamation}.

\begin{defn}[Lift $L_0$]
\label{def:order}
Given a good bowtie-free graph $G=(V, E)$, an {\em ordered good bowtie-free graph} is a structure
$\str{G}=(V,\rel{G}{E_0},\rel{G}{E_1},\leq_\str{G})$ where
$\rel{G}{E_0}=E_0(G)$, $\rel{G}{E_1}=E_1(G)$ and $\leq_\str{G}$ is a linear
order of $V$ such that
\begin{enumerate}
\item vertices of every centre of every vertex $v\in G$ form an interval  of $\leq_\str{G}$,
\item all centres of chimneys are before vertices in copies of $K_4$, 
\item all central vertices are before non-central vertices, and
\item 
non-central vertices belonging
to a given chimney form an interval. 
The relative order of these intervals
corresponding to given centres follows the order of the relative order of the centres.
\end{enumerate}

Such ordering is called an {\em admissible ordering}.
We denote by $L_0$ the language of ordered good bowtie-free graphs and by $\B_0$ the class
of all ordered good bowtie-free  graphs. By an abuse of notation, for a good bowtie-free graph $G$ we also denote $\str{G} = L_0(G)$ the corresponding ordered good bowtie-free graph (i.e.  $\str{G}$ is an $L_0$-lift of $G$). 
\end{defn}

\begin{example}
One of admissible orderings of our reference graph in Figure~\ref{fig:bowtie-example} is:
$u^1_\ell$, $u^1_r$, $u^2_\ell$, $u^2_r$, $u^3_\ell$, $u^3_r$, $u^4_1$, $u^4_2$, $u^4_3$, $u^4_4$, $v^1_1$, $v^1_2$, $v^1_3$, 
$v^2_1$, $v^2_2$,  $v^3_1$, $v^3_2$, $v^3_3$.
There are four centres of a vertex in the graph: $\{u^1_\ell, u^1_r\}$,  $\{u^2_\ell, u^2_r\}$,  $\{u^3_\ell, u^3_r\}$, and $\{u^4_1, u^4_2, u^4_3, u^4_4,\}$ each
of them forms an interval (to satisfy 1). $\{u^4_1, u^4_2, u^4_3, u^4_4,\}$ is after all centres of vertices in a chimney (to satisfy 2). The non-central vertices
associated with each chimney forms an interval: $\{v^1_1, v^1_2, v^1_3\}$, $\{v^2_1, v^2_2\}$, and $\{v^3_1, v^3_2, v^3_3\}$ and their relative order corresponds to the order of their centres, as required by 4.
\end{example}

We introduce two more lifts of good bowtie-free graphs. The 
lift $L_1$ is introducing unary relations and the lift $L_2$ in addition binary relations. It is $L_0\subset L_1 \subset L_2$.
The hereditary class defined by the lift $L_2$  will form our Ramsey class.

\begin{defn}[Lift $L_1$]
Let $\str{G}$ be an ordered good bowtie-free graph.  $\str{A}=L_1(\str{G})$ is a
lift of $\str{G}$ adding new unary relations $\rel{A}{\ell}$, $\rel{A}{r}$, $\rel{A}{1}$,
$\rel{A}{2}$, $\rel{A}{3}$ and $\rel{A}{4}$ such that:
\begin{enumerate}
 \item for every pair $u,v$ forming the centre of a chimney of $\str{G}$, $u<_\str{G} v$,
we put $(u)\in\rel{A}{\ell}$ and $(v)\in\rel{A}{r}$;
 \item for every $a<_\str{G} b<_\str{G} c<_\str{G} d$ that are vertices of a copy of $K_4$ in $\str{G}$
we put $(a)\in\rel{A}{1}$, $(b)\in\rel{A}{2}$, $(c)\in\rel{A}{3}$, $(d)\in\rel{A}{4}$.
\end{enumerate}
We denote by $L_1$ the language of this lift. For a given $\str{G} \in \B_0$ we denote by $L_1(\str{G})$ the corresponding lift of $\str{G}$.
By $\B_1$ we denote the class of all structures $L_1(\str{G})$, $\str{G}\in \B_0$.
\end{defn}

\begin{example}
The unary relations of the $L_1$-lifts of our reference graph in Figure~\ref{fig:bowtie-example} 
are indicated by labels of the $u$ vertices.
\end{example}

Advancing the definition of  the lift $L_2$ we first note that we shall sometimes consider {\em rooted} structures (with 
either one or two {\em roots}).
Isomorphisms (and embeddings) are, of course, defined as root preserving isomorphisms (and embeddings).
If, for example, the structures $\str{G}$ and $\str{G}'$ are considered with roots $u$ and $v$  and $u'$ and $v'$   then these structures are called isomorphic if there is an isomorphism $f$ from  $\str{G}$ to $\str{G}'$  such that $f(u) = u'$  and $f(v) = v'$.  

Given ordered good bowtie-free graph $\str{G}$ and two vertices $u$, $v$, $u\neq v$, we denote by $t(u,v)$ the isomorphism type of the structure
induced by $L_1(\str{G})$ on the set $\{u,v\}\cup c(u) \cup c(v)$ rooted in $(u,v)$.
We fix an enumeration $t_1,t_2,\ldots, t_N$ of all the possible such types. Clearly there are only finitely many possibilities for types (as $t(u,v)$ is an isomorphism type of a graph with at most $8$ vertices).
 In this situation we define binary relations $\rel{}{t_1},\rel{}{t_2},\ldots, \rel{}{t_N}$ as follows:

\begin{defn}[Homogenising lift $L_2$]
\label{defn:L2}
Let $\str{G} \in \B_0$ be an ordered good bowtie-free graph.  $\str{A}=L_2(\str{G})$ is the
lift consisting from  $L_1$-struc\-ture $L_1(\str{G})$ and in addition  from new binary relations
$\rel{A}{t_1},\rel{A}{t_2},\ldots, \rel{A}{t_N}$. For $u<_\str{G} v$ we put $(u,v)\in \rel{A}{t(u,v)}$.
We denote by $L_2$ the language of this lift and by $\B_2$ the class of all structures
$L_2(\str{G})$, $\str{G}\in \B_0$.

If $(u,v)\in \rel{A}{t(u,v)}$ then $t(u,v)$ is called the {\em type} of pair $(u,v)$.
\end{defn}
Denote by $\mathcal B_2$ the class of all $L_2$-lifts $L_2(\str{G})$, $\str{G}\in \B_0$. 
Observe that structures in $\mathcal B_2$ have the property that every two distinct vertices are in a tuple of some relation (called irreducible below).

\begin{example}
\begin{figure}
\centerline{\includegraphics{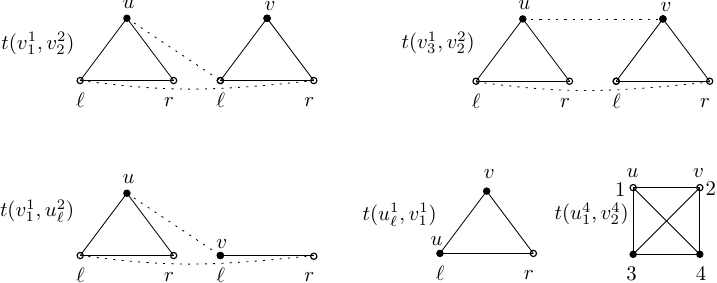}}
\caption{Examples of types of pairs appearing in our reference graph depicted in Figure~\ref{fig:bowtie-example}. (With the notation for vertices preserved.)}
\label{fig:types}
\end{figure}
Some types of pairs in our reference graph depicted in Figure~\ref{fig:bowtie-example} are depicted in Figure~\ref{fig:types}.
\end{example}
\begin{remark}
The relations $\rel{A}{t_i}$ of the lift $L_2(\str{G})$ form a natural homogenisation of ordered good bowtie-free graphs
as the new binary relations introduced describe necessary orbits of the
automorphism group of a universal graph for class $\B_1$.  On the other hand
not all relations in lift $L_1(\str{G})$ (i.e. unary relations) are necessary from the point of view
of $\omega$-categoricity.  The universal graph in~\cite{Cherlin1999} has automorphisms exchanging vertices within 
centres of vertices.  We however consider ordered graphs and the order on every 
centre prevents any non-trivial automorphism within it.  It is also interesting to
observe that~\cite{Cherlin1999} leads to homogenisation of the existentially complete universal graph which needs relations of unbounded arity.  Our ordered lift has only unary
and binary relations. This is in agreement with \cite{Hartman2014} where we show that the relational
complexity of $\omega$-categorical and existentially complete bowtie-free graph
is infinite, while the relational
complexity of the ordered $\omega$-categorical and existentially complete bowtie-free graph
is 2.
\end{remark}

\begin{defn}
\label{def:L2}
Denote by $\overline{\B}$ the class of all substructures of $\B_2$. (Thus $\overline{\B}$ is the hereditary closure of $\B_2$.) 
For structure $\str{A}\in \overline{\B}$ an ordered good bowtie-free graph
$\str{G}\in \B_0$ is called a {\em witness} of $\str{A}$ if $\str{A}$ is induced on
$A$ by $L_2(\str{G})$.
\end{defn}
It follows directly that $\str{A}\in \overline{\B}$ if and only if there exists a witness $\str{G}$ of $\str{A}$.

The lift $L_2(\str{G})$ encodes enough information so that for every
substructure $\str{A}$ of $L_2(\str{G})$ it is possible
to uniquely reconstruct the type of its centre (a precise procedure for this appears in proof of Theorem~\ref{thm:triangles}). Lemma \ref{lem:amalgamation} extends to the amalgamation property of $\overline{\B}$:
\begin{lem}[Amalgamation of lifts]
\label{lem:amalg2}
$\overline{\B}$ is an amalgamation class.  
\end{lem}
\begin{proof}
Fix $\str{A}$, $\str{B}_1$ and $\str{B}_2$ from $\overline{\B}$ such
that identity is an embedding from $\str{A}$ to $\str{B}_1$ and $\str{B}_2$.
We will construct an amalgamation of $\str{B}_1$ and $\str{B}_2$ over $\str{A}$.

Let $\str{G}_{\str{B}_1}$ and $\str{G}_{\str{B}_2}$ be witnesses of
$\str{B}_1$ and $\str{B}_2$ respectively.
 Further denote by $\str{A}_1$ the structure
induced by $\str{G}_{\str{B}_1}$ on the union of $A$ and all centres
of vertices of $A$ in $\str{G}_{\str{B}_1}$.
Similarly denote by $\str{A}_2$ the structure
induced by $\str{G}_{\str{B}_2}$ on the union of $A$ and all 
centres of vertices of $A$ in $\str{G}_{\str{B}_2}$. 

By the construction of the lift, $\str{A}_1$ is isomorphic to $\str{A}_2$
and moreover there is an isomorphism that is an identity
on $A$. 
It is now possible to extend
$\str{G}_{\str{B}_1}$ to $\str{G}'_{\str{B}_1}$ and $\str{G}_{\str{B}_2}$
to $\str{G}'_{\str{B}_2}$ (by possibly adding  centres) in a way that centres $c(\str{G}'_{\str{B}_1})$ and
$c(\str{G}'_{\str{B}_2})$ are isomorphic with fixing vertices of
$\str{A}$.

By the same argument as in proof of Lemma \ref{lem:amalgamation} we get ordered good bowtie-free graph
$\str{G}_\str{D}$ that is an amalgamation of  $\str{G}'_{\str{B}_1}$ and
$\str{G}'_{\str{B}_2}$ over $\str{A}_1=\str{A}_2$ (here we may unify non-central vertices in $\str{A}$, too).  It is easy to verify
that $L_2(\str{G}_\str{D})$ is as well an amalgamation of $\str{B}_1$ and
$\str{B}_2$ over $\str{A}$, since the type of every pair of vertices in
$\str{B}_1$ or $\str{B}_2$ is preserved and thus $\str{B}_1$ and
$\str{B}_2$ are induced substructures of $L_2(\str{G}_\str{D})$.
\end{proof}

Thus we may apply \Fraisse{}-theory (see Section~\ref{remarks}). To obtain Ramsey property we
apply strong Ramsey properties proved in \cite{Nevsetvril1977b}. Towards this end we need the following alternative description of $\overline{\B}$ by means of forbidden substructures.

Recall that a structure $\str{A}$ is called {\em irreducible} if every pair of distinct vertices belong to a relation of $\str{A}$.
In the context of Ramsey theory it is often convenient to consider the lift
adding linear order alone.
Let $K$ be a language and $K^\leq$ lifted language adding single binary relation $\leq$.
$K^\leq$-structure $\str{A}$ is {\em pure-irreducible} if its $K$-shadow is irreducible.

We will apply notion of pure-irreducibility to structures in $\str{A}\in \overline{\mathcal {B}}$.
While the linear order is present in $\sh(\str{A})$ implicitly this will
allow us to describe class $\overline{\mathcal {B}}$ by means of forbidden pure-irreducible
substructures.

Given family of finite $L$-structures $\F$ we denote by $\Forb(\F)$ the class of finite $L$-structures not containing any structure $\str{F}\in \F$ as a substructure.
We sometimes write $\Forb_L(\F)$ to denote explicitly the language $L$ of structures we are considering.

\begin{thm}
\label{thm:triangles}
Let $\mathcal T$ be the class of all pure-irreducible $L_2$-structures $\str{F}$ with at most 3 vertices such that $\str{F}\notin {\overline \B}$.
Then\begin{enumerate}
 \item every pure-irreducible structure in $\Forb_{L_2}(\mathcal T)$ is lift of a bowtie-free graph, and,
 \item the class $\overline {\mathcal B}$ is precisely the class of all finite pure-irreducible structures $\str{A}$ in $\Forb_{L_2}(\mathcal T)$ such that $\leq_\str{A}$ is
 an admissible ordering.
\end{enumerate}
\end{thm}

\begin{proof}
Consider structure $\str{A}\in \overline{\B}$.  It easily follows from
Definition \ref{def:L2} that every pair of vertices $(u,v)$, $u\leq_\str{A}
v$, is in some binary relation $\rel{A}{t_j}$ and thus $\overline{\B}$ is
a class of pure-irreducible structures.

Because $\overline{\B}$ is closed on  substructures
it remains to show that every pure-irreducible structure $\str{A}\notin \overline{\B}$
contains substructure $\str{A}'\notin \overline{\B}$ that consists of at
most 3 vertices.

We give an effective procedure that attempts to construct, for a given ordered structure $\str{A}$, an ordered good
bowtie-free graph (a witness) $\str{G}$  such that
$\str{A}$ is an induced substructure of $L_2(\str{G})$.
The existence of the witness $\str{G}$ proves that $\str{A}\in \overline{\B}$.
Then we analyse cases where such procedure fails and show that
these failures all correspond to structures $\str{F}$ on at most 3~vertices.
All those structures will have property that $\str{F}\notin \overline{\B}$.

Denote by $\str{A}^0$ the $L_1$-shadow of $\str{A}$.  Enumerate all
pairs of vertices $u,v$, $u<_\str{A} v$ in $A$ as $(u_1,v_1),(u_2,v_2),\ldots, (u_n,v_n)$.
For every pair $(u_i,v_i)$, $1\leq i\leq n$, we construct $\str{A}^i$
inductively from $\str{A}^{i-1}$ based on the type of $(u_i,v_i)$ in
$\str{A}$.  This involves the following elementary steps:
\begin{enumerate}
\item addition of new vertices to represent centres of $u$ and $v$ if they are not already present in $\str{A}^{i-1}$,
\item addition of the new vertices into the corresponding unary relations $\nbrel{\str A^i}{\ell}$, $\nbrel{\str A^i}{r}$, $\nbrel{\str A^i}{1}$,
$\nbrel{\str A^i}{2}$, $\nbrel{\str A^i}{3}$, and, $\nbrel{\str A^i}{4}$ as required by the type,
\item addition of new edges of type 0 or 1 from $u$ and $v$ to the newly added vertices,
\item addition of edges of type 0 or 1 between the newly added vertices,
\item extension of the linear order $\leq_{\str{A}^{i-1}}$ in a way consistent
with the definition of ordered good bowtie-free graphs (Definition~\ref{def:order}) and the type of the pair $(u,v)$.
\end{enumerate}
Because vertices of centres of a given vertex $v$ are uniquely determined by
the unary relations in $L_1$, there is (up to isomorphism) unique way of doing so (if it exist at all).

This procedure may fail if the extension is impossible. Assume that $(u_i,v_i)$
is the first pair such that $\str{A}^i$ can not be constructed. We consider
individual cases that may happen and show that such failure scenarios all imply
existence of forbidden substructures in $\str{A}$ with at most 3 vertices:
\begin{itemize}
 \item[1.] Some or all vertices of the centre of $u_i$ already exists in
$\str{A}^{i-1}$ and they are in conflict with the centre required by the
type of pair $(u_i,v_i)$.

For example there is a vertex $u'$ connected by edge of type 0 to $u_i$
that is in $\rel{A}{\ell}$ while the centre of $u_i$ required is a copy of $K_4$ and thus
the vertex should be in $\rel{A}{1}$, $\rel{A}{2}$, $\rel{A}{3}$ or $\rel{A}{4}$ instead.

In this case let $u'$ be such vertex. If $u'$ is in $A$ then the structure induced
on $u_i,v_i,u'$ must be forbidden: pair $(u_i,v_i)$ require $u_i$ to have its centre of one
type, while pair $(u_i,u')$ require its centre of a different type (or if $u_i=u'$ then
unary the relation on $u_i$ must be already in conflict).  This is not possible in a
structure in $\overline{\B}$.

If $u'$ is not in $A$ then it was introduced when defining the centre of vertex
$u''$ and then $u_i,v_i,u''$ induce the forbidden substructure for the same reason.

 \item[2.] The centre of $v_i$ is already present in the structure and differs from one required by the type.

This case follows in complete analogy to 1.

 \item[3.] Vertices $u_i$ and $v_i$ are connected or ordered differently than
required by the type. 
In this case the structure induced on $u_i,v_i$ is forbidden.

 \item[4.] Edges or orders in between already defined parts of centres $u_i$ and
$v_i$ are different then required by the type. 

Denote by $u'$ and $v'$ the
conflicting vertices of the centre of $u_i$ and $v_i$ respectively.  Now put $u''=u'$
if $u'\in A$ or put $u''$ to be a vertex of $A$ whose centre contains $u'$.
Similarly put $v''=v_i$ if $v'\in A$ or $v''$ to a vertex of $A$ whose centre
contains $v'$.  Now structure induced on $u_i,v_i,u'',v''$ is forbidden and
moreover at least one of $u_i,v_i,u''$ or $u'',v_i,v''$ must be forbidden.
\end{itemize}
We have shown that the procedure of adding centres can always be completed for all pairs $(u_i, v_i), i = 1,2,\ldots,n$  if
all substructures of $\str{A}$ on at most 3 vertices are in $\overline{\B}$. It also follows that if $\str{A}^i\in \Forb(\mathcal T)$ then also $\str{A}^{i+1}\in \Forb(\mathcal T)$ for every $1\leq i<n$.
Denote by $\str{G}^n$ the resulting ordered graph (i.e. $L_0$ shadow of $\str{A}^n$).

In the final step we construct $\str{G}$ by extending every triangle in $\str{G}^n$
that is not contained in a chimney nor $K_4$ and every edge of type 0 not contained in a triangle to a copy of chimney $Ch_2$.  There is unique
way of doing so that is consistent with lift $\str{A}^n$:
\begin{enumerate}
\item Every isolated triangle in $\str{A}^n$
must contain precisely one vertex $v_\ell$ in relation $\rel{A}{\ell}$, one vertex $v_r$ in relation $\rel{A}{r}$ and one vertex $v_t$ in no unary relations.
(Any other triangle is either forbidden by $\mathcal T$ or contains vertices in relations
 $\rel{A}{1}$, $\rel{A}{2}$, $\rel{A}{3}$ or $\rel{A}{4}$ and those was extended
to copies of $K_4$ during the construction of $\str{G}^n$.) The extension of such triangle
can thus only be done by adding a new vertex $v$ adjacent to both $v_\ell$ and $v_r$.
The order of $\str{G}$ can be extended by putting $v$ just after $v_t$ to satisfy Definition \ref{def:order} condition 4. (this is the only step of the construction that is not unique and it would be also possible to put $v$ before $v_t)$.
\item Every edge of type 0 has precisely one vertex  $v_\ell$ in relation $\rel{A}{\ell}$ and one vertex $v_r$ in relation $\rel{A}{r}$ (again any other edge of type 0 must have been extended to a copy of $K_4$). There is only one way to extend the order which is consistent with Definition \ref{def:order} condition 4.
\end{enumerate}

We shall verify that $\str{G}$ is an ordered good bowtie-free graph.  While constructing $\str{A}^n$ we made sure that
every vertex is contained in at least one edge of type 0. While constructing $\str{G}$ we made sure that
every such edge is contained either in a chimney or $K_4$. Because $\str{A}^n\in \Forb(\mathcal T)$ and only edges contained in triangles were added to $\str{G}$ we know
that all edges contained in a triangle are of type 0 and the subgraph formed by edges of type 0 is a disjoint union
of chimneys and $K_4$'s.

By Lemma \ref{lem:bowtiefree} we get that $\str{G}$ is a bowtie-free graph. By the construction
$\str{G}$ is a witness of $\str{A}^n$ and thus also of $\str{A}$.

It follows that we can characterise lifts $\str{A}$ such that there exists
$\str{G}$ (described above) that is an ordered good bowtie-free graph.
Because $\str{A}$ is a substructure of $L_2(\str{G})$ (and thus
$\str{A}\in \overline{\B}$) the statement follows.
\end{proof}

\section{Reduced structures}
\label{sec:reduced}
To simplify our future analysis, we now invoke another modification of $L_2$-struc\-tures. It is easy to see that
the $L_2$-lift was created in such a  way that  all edges in between two centres of a vertex, $C_1$ and $C_2$, are in fact encoded by type of any pair of vertices $v_1\in C_1$ and $v_2\in C_2$ which is explicitly represented in $L_2$-lifts.  We can thus safely omit all but one vertex from every centre of a vertex without losing
any information about a good bowtie-free $L_2$-struc\-ture:
\begin{defn}[Reduced structures]
$\str{A}^\bullet$ is a {\em reduction} of $L_2$-struc\-ture $\str{A}$
if it is created from $\str{A}$ by removing all vertices $v\in \rel{A}{r},\rel{A}{2},\rel{A}{3},\rel{A}{4}$.  
\end{defn}
 We modify the language $L_2$
correspondingly into $L^\bullet$ (actually we may take $L^\bullet=L_2$ but relations $\rel{}{r},\rel{}{2},\rel{}{3},\rel{}{4}$ are always empty for reduced structures and thus we remove them from the language) and denote by $\B^\bullet$ the class of all reduced
structures $\str{A}^\bullet$ where $\str{A}\in \B_2$ and by $\overline{\B}^\bullet$ for all reduced structures $\str{A}^\bullet$ where $\str{A}\in \overline{\B}$. Accordingly we modify the other definitions (such as the definition of pure-irreducible structures).

By comparing the corresponding definitions we have that reduced  structures are
still described by a set of forbidden substructures with at most 3 vertices (in a
sense of Theorem~\ref{thm:triangles}):
\begin{thm}
\label{thm:triangles2}
$\overline{\mathcal B}^\bullet$ is the class of all finite admissibly ordered pure-irreducible structures in $\Forb_{L^\bullet}(\mathcal T^\bullet)$ where $\mathcal T^\bullet$ is a finite set of pure-irreducible structures with at most 3 vertices.
\end{thm}
\begin{proof}
$\mathcal T^\bullet$ is the subset of $\mathcal T$ defined in Theorem~\ref{thm:triangles} containing $L^\bullet$ shadows of all structures $\str{A}\in \mathcal T$
such that all relations $\rel{A}{r}$, $\rel{A}{2}$, $\rel{A}{3}$, $\rel{A}{4}$ are empty.
\end{proof}

\section{Ramsey structures}
\label{sec:ramsey}

The following strong Ramsey theorem is a variant of the main result of \cite{Nevsetvril1977b}, see also \cite{Nevsetvril1995}. It will be used repeatedly (for example in Sections \ref{sec:star}  and \ref{sec:final}). 
\begin{thm}[\cite{Nevsetvril1977b}]
\label{thm:NR0}
Let $K$  be a finite relational language involving binary relation $\leq$ and $\F$ be a set of pure-irreducible $K$-struc\-tures. Then the class of all structures in $\Forb(\F)$ where $\leq$ is a linear ordering of vertices is a Ramsey class.
\end{thm}
Note that traditionally Theorem~\ref{thm:NR0} is formulated in the language of linearly ordered hypergraphs. Because of reducibility of linearly ordered hypergraph is concerned about hyperedges only, it corresponds to our notion of pure-irreducibility.
We will use the following refinement of Theorem~\ref{thm:NR0}.

\begin{thm}[\cite{Nevsetvril1977b}]
\label{thm:NR}
Let $K$  be a finite relational language involving binary relation $\leq$ and unary relations $U_1, U_2,\ldots, U_N$.
Let $\F$ be a set of pure-irreducible $K$-struc\-tures. Let $\C$ be the class of all $K$-struc\-tures $\str A$ in $\Forb(\F)$
where every vertex is in exactly one of unary relations $\rel{A}{U_i}$ and
$\leq_{\str{A}}$ is linear ordering of vertices which satisfies 
$$
x < y  \hbox { whenever }  (x)\in \rel{A}{U_i} \hbox { and } (y)\in \rel{A}{U_j} \hbox { and } 1\leq i<j\leq N.
$$
We call such ordering an {\em admissible ordering}.
Then the class $\mathcal C$ together with admissible orderings is a Ramsey class.
\end{thm}
\begin{proof}
In fact, this is an $N$-partite version of the main result of \cite{Nevsetvril1977b}. It follows directly by a product argument.
For completeness, this can be outlined as follows: Given admissibly ordered $\str{A},\str{B}\in \mathcal C$, by~\cite{Nevsetvril1977b} there exists $\str C\in\Forb(\F)$ with $\str C\longrightarrow (\str B)^\str A_2$. $\str{C}$ is ordered but this ordering $\leq_\str{C}$ may not have to be admissible. Then it is possible to re-order vertices of $\str{C}$ lexicographically first by unary relations they belong to and second by the original order of $\str{C}$. It is easy to see that this new order is admissible and preserves all copies of (admissibly ordered)  structure $\str{B}$.
\end{proof}
 We apply Theorem~\ref{thm:NR} to  the set $\mathcal T^\bullet$ of pure-irreducible $L^\bullet$-structures defined in Theorem  \ref{thm:triangles2}.
 We consider the structures in  $\Forb_{L^\bullet}(\mathcal T)$ with admissible orderings defined above.
 Theorem \ref{thm:NR} then specialises to the following:

\begin{thm}\label{mainI}
The class $\Forb_{L^\bullet}(\mathcal T^\bullet)$   is a Ramsey class. 

Explicitly: 
For every pair of $L^\bullet$-struc\-tures $\str{A}, \str{B}$ in  $\Forb_{L^\bullet}(\mathcal T^\bullet)$ there exists an
 $L^\bullet$-structure $\str{C} \in \Forb_{L^\bullet}(\mathcal T^\bullet)$  such that
$$
\str{C} \longrightarrow (\str{B})^{\str{A}}_2.
$$
\end{thm}

\begin{proof}
Indeed this is just a specialisation of Theorem~\ref{thm:NR} (where admissible orderings are interpreted by orderings of chimneys and $K_4$'s and language is extended by additional unary relation containing precisely those vertices not in any of unary relations of $L^\bullet$)).
\end{proof}

However note that even when $\str{A}$ and $\str{B}$ are pure-irreducible structures in $\Forb_{L^\bullet}(\mathcal T^\bullet)$ the structure $\str{C}$ in  $\Forb_{L^\bullet}(\mathcal T^\bullet)$ is not necessarily pure-irreducible and thus it may not correspond to the reduction of $L_2$-lift of a good bowtie-free graph. (As there are forbidden configurations we cannot complete $\str{C}$ to a pure-irreducible structure ``freely''.) 

\section{Star Equivalences are Ramsey}
\label{sec:star}

The key feature of bowtie-free graphs is the partition to chimneys with each class of the partition ``rooted'' in the centre (the root being its algebraic closure).
In this section we prove Theorem~\ref{thm:mainII} which extends Theorem~\ref{mainI} to structures with such ``rooted equivalences''. This brings us closer to the main result (which is proved in the next section).

\begin{defn}[Chimney equivalence]
For a $L^\bullet$-struc\-ture $\str{A} \in {\B^\bullet}$ (i.e. which is the reduction of the $L_2$-lift of a good bowtie-free graph)  denote by $\sim_{\str{A}}$ the equivalence expressing that two vertices belong to the same chimney (contracted central vertices are included in this). 
$\sim_{\str{A}}$  is called the {\em chimney equivalence} of $\str{A}$.
\end{defn}

Note that each equivalence class of $\sim_{\str{A}}$ contains a distinguished vertex $x$ which is the (reduced) centre of the corresponding chimney or a copy of $K_4$. Moreover all other vertices of this equivalence class are related to $x$ by edges belonging to 
$\rel{A}{E_0}$ that corresponds to a (spanning) star  and there are no other vertices joined to $x$ by $\rel{A}{E_0}$ edges.
Thus the equivalence $\sim_{\str{A}}$ is described by a star forest formed by $\rel{A}{E_0}$ edges.
({\em Star} is a complete bipartite graph $K_{1,k}$, $k\geq 0$. 
Thus a single vertex is also a star. If $K=1$ the star is an edge and the unique vertex in $\rel{A}{\ell}$ or $\rel{A}{1}$ considered as root.
{\em Star forest} is any graph created as a disjoint union of stars.)
This leads us to the following definition which makes sense for structures in 
$\Forb_{L^\bullet}(\mathcal T^\bullet)$ which are not necessarily irreducible:

\begin{defn}[Star equivalence]
For an $L^\bullet$-struc\-ture $\str{A} \in \Forb_{L^\bullet}(\mathcal T^\bullet)$ assume that the edges $\rel{A}{E_0}$ form a star forest.   Denote by $\approx_{\str{A}}$ (called {\em star equivalence}) the equivalence expressing the component structure of this star forest. 
\end{defn}

The equivalence $\approx_{\str{A}}$ for structures that are not necessarily pure-irreducible will play the role of the chimney equivalence for pure-irreducible structure.

\begin{defn}
Denote by $\Forb_{L^\bullet}^{\approx}(\mathcal T^\bullet)$ the class of all (not necessarily pure-irreducible)  structures $\str{A} \in \Forb_{L^\bullet}(\mathcal T^\bullet)$ where $\approx_{\str{A}}$ is a star equivalence (that is edges $\rel{A}{E_0}$ forms a star forest) and such that all vertices that appear in centres of stars (possibly degenerated to 1 vertex) are either in $\rel{A}{\ell}$ or $\rel{A}{1}$. 
\end{defn}

In this section we  aim to prove Theorem~\ref{thm:mainII} which gives Ramsey property for structures with
star equivalences. Advancing this we  modify the key part of proof of Theorem~\ref{thm:NR}.

We shall stress the fact that $\Forb_{L^\bullet}^{\approx}(\mathcal T^\bullet)$ can not be expressed
as a class $\Forb_{L^\bullet}(\mathcal T')$ where $\mathcal T'$ is a set of
pure-irreducible structures.  There is no way to express the fact that no vertex can be connected
to centres of two different stars. Consequently we can not apply Theorem~\ref{thm:NR} (or \ref{mainI}) directly.
The proof below uses a variant of the Partite Construction \cite{Nevsetvril1989}. We modify its core part---Partite Lemma---in order to satisfy the 
additional equivalence condition.

The following is the main definition of this section.
\begin{defn}[$\str{A}$-partite structure]
Let $\str{A}$ be a  $L^\bullet$-structure. Assume $A = \{1, 2,\ldots, a\}$ with the natural ordering.  An $\str{A}$-partite $L^\bullet$-structure is a tuple $(\str{A},{\mathcal X}_\str{B},$ $\str{B})$ 
where $\str{B}$ is an $L^\bullet$-struc\-ture and $\mathcal X_\str{B}=\{X^1_\str{B},X^2_\str{B},\ldots, X^a_\str{B}\}$  partitions vertex set of $\str{B}$ into $a$ classes ($X^i_\str{B}$  are called {\em parts} of $\str{B}$)  such that 
\begin{enumerate}
\item the ordering of $\str{B}$ is lexicographic induced by the ordering of $\str{A}$ and of parts $X^i_\str{B}$ (particularly it satisfies $X^1_\str{B} < X^2_\str{B} < \ldots < X^a_\str{B}$);

\item mapping $\pi$ which maps every $x \in X^i_\str{B}$ to $i$ $(i = 1,2,\ldots,a)$ is a homomorphism $\str{B}\to\str{A}$ in $L^\bullet$ ($\pi$ is called the {\em projection}); 

\item every tuple in every relation of $\str{B}$ meets every class $ X^i_\str{B}$ in at most one element.
\end{enumerate}
The isomorphisms and embeddings of $\str{A}$-partite structures, say of $\str{B}$ into $\str{B'}$ are defined as the isomorphisms and embeddings of $L^\bullet$-structures together with the condition that all parts 
are preserved (i.e. the part $X^i_\str{B}$ is mapped to $X^i_\str{B'}$  for every $i = 1,2,\ldots,a$).
\end{defn}

In the following we will consider $L^\bullet$-structure $\str{A}$ to be also an $\str{A}$-partite structure, where each class of the partitions $X^1_\str{A}, X^2_\str{A},\ldots, X^a_\str{A}$ consists of single vertex. For brevity, given a class of $L^\bullet$ structures $\mathcal K$ and an $L^\bullet$-partite structure $\str{B}=(\str{A},{\mathcal X}_\str{B},$ $\str{B}')$, we will also write $\str{B}\in \mathcal K$ with the meaning $\str{B}'\in \mathcal K$.
We start by proving the following modification of the Partite Lemma \cite{Nevsetvril1989}. The main difference is that we consider structures with equivalences.

\begin{lem}[Partite Lemma]\label{partlem}
Let $\str{A}\in \B^\bullet$ be an $L^\bullet$-struc\-ture with star equivalence $\approx_{\str{A}}$ induced by $\rel{A}{E_0}$ edges. Assume without loss of generality  $A = \{1, 2,\ldots, a\}$ with the natural ordering. Let $\str{B}\in \Forb_{L^\bullet}^{\approx}(\mathcal T^\bullet)$  be an $\str{A}$-partite $L^\bullet$-structure
with parts $\mathcal X_\str{B} = \{X^1_\str{B},X^2_\str{B},\ldots, X^a_\str{B}\}$  and star equivalence $\approx_{\str{B}}$ induced by 
$\rel{B}{E_0}$ edges. Further assume that every vertex of $\str{B}$ is contained in a copy of $\str{A}$.
Then there exists an $\str{A}$-partite $L^\bullet$-structure $\str{C}$
with parts $\mathcal X_\str{C} = \{X^1_\str{C},X^2_\str{C},\ldots, X^a_\str{C}\}$ where $\rel{C}{E_0}$ form a star forest defining star equivalence $\approx_{\str{C}}$ 
such that 
$$
\str{C}\longrightarrow (\str{B})^\str{A}_2.
$$

Explicitly: For every $2$-colouring of all $\str{A}$-partite substructures of $\str{C}$ which are isomorphic to $\str{A}$ there exists a substructure $\widetilde{\str{B}}$ of $\str{C}$, $\widetilde{\str{B}}$ isomorphic to $\str{B}$, such that all the substructures of $\widetilde{\str{B}}$ which are isomorphic to $\str{A}$ are all monochromatic.
Particularly, the isomorphism of $\widetilde{\str{B}}$ and $\str{B}$ (which is an embedding of $\str{B}$ into $\str{C}$) and thus maps $\rel{B}{E_0}$ to $\rel{C}{E_0}$ and therefore also maps the equivalence  $\approx_{\str{B}}$   to  $\approx_{\str{C}}$. 
\end{lem}

\begin{proof}
Let $\widetilde{\str{A}}_1, \widetilde{\str{A}}_2,\ldots, \widetilde{\str{A}}_t$ be the enumeration of all substructures of $\str{B}$ which are isomorphic to $\str{A}$. 

We take $N$ sufficiently large (that will be defined later) and construct an $\str{A}$-partite $L^\bullet$-structure  $\str{C}$
with parts $\mathcal X_\str{C} = \{X^1_\str{C},X^2_\str{C},\ldots, X^a_\str{C}\}$  as follows:
\begin{enumerate}
\item For every $1\leq i\leq a$ set $X_\str{C}^i$ is the set of all functions $$f:\{1,2,\ldots,N\}\to X_\str{B}^i.$$
\item 
The ordering $\leq_\str{C}$ of $\str{C}$ is defined  lexicographically as an extension of all orderings of $\str{A}$ and $\{1,2,\ldots,N\}$.
\item For every relational symbol $\rel{}{}\in L^\bullet$, $(f_1,f_2,\ldots, f_r)\in \rel{C}{}$ if and only if one of the following occurs:
\begin{enumerate}
\item There exists function $u:\{1,2,\ldots, N\}\to \{1,2,\ldots, t\}$ such that for every $1\leq l\leq N$ the tuple $(f_1(l),f_2(l),\ldots, f_r(l))$ is in $\nbrel{\widetilde{\str{A}}_{u(l)}}{}$.
\item There exists $\omega\subseteq \{1,2,\ldots, N\}$ and function $u:\{1,2,\ldots, N\}\setminus \omega\to \{1,2,\ldots, t\}$ such that functions $f_1, f_2, \ldots, f_r$ are all constant on $\omega$ and:
\begin{enumerate}
\item $f_1(l)$, $f_2(l), \ldots, f_r(l)$ are all vertices of $\widetilde{\str{A}}_{u(l)}$ in $\str{B}$, for every $l\in \{1,2,\ldots, N\}\setminus \omega$,
\item $(f_1(l),f_2(l),\ldots, f_r(l))\in \rel{B}{j}$ and there is no copy of $\str{A}$ in $\str{B}$ containing all vertices $f_1(l),f_2(l),\ldots, f_r(l)$, for $l\in \omega$.
\end{enumerate}
\end{enumerate}
\end{enumerate}

If vertex $v$ of $\str{B}$ is contained in the star $S$, denote by $s_\str{B}(v)$ the centre of $S$ and put $s_\str{B}(v)=v$ otherwise.
Define the equivalence $\equiv$ on $\str{C}$ as follows:
$f\equiv g$ if and only if $s_\str{B}(f(l))=s_\str{B}(g(l))$ for every $1\leq l\leq N$.

Observe that $\str{C}\in\Forb_{L^\bullet}(\mathcal T^\bullet)$ as $\str{A}\in \Forb_{L^\bullet}(\mathcal T^\bullet)$ and the projection $\pi:\str{C}\to \str{A}$ is a homomorphism.
Moreover every pure-irreducible structure $\str{D}$ of $\str{C}$ contains at most one vertex from every partition $X_\str{C}^i$ (by the construction) and thus $\pi$ restricted to $\str{D}$ is injective.

We shall check that indeed $\str{C}$ is  an $\str{A}$-partite $L^\bullet$-structure (and thus again $\str{C} \in \Forb_{L^\bullet}(\mathcal T^\bullet)$)
with parts $\mathcal X_\str{C} = \{X^1_\str{C},X^2_\str{C},\ldots, X^a_\str{C}\}$  and finally we prove that the edges $\rel{C}{E_0}$ form a star forest and that the star equivalence $\approx_{\str{C}}$ coincides with  $\equiv$.
Most of this follows immediately from the definition. We pay extra attention to checking that $\equiv$ give the star equivalence. This is the main difference from \cite{Nevsetvril1989}.

It is easy to see that $\equiv$ is indeed an equivalence.
We show that the $\equiv$ is the star equivalence of  $\approx_{\str{C}}$ induced by edges $\rel{C}{E_0}$. First observe that $u\approx_\str{B} v$ if and only if $s_\str{B}(u)=s_\str{B}(v)$. Moreover, because $\str{A}$ corresponds to a reduced good bowtie-free structure and because every vertex of $\str{B}$ is contained in a copy of $\str{A}$ we know that every edge of type 0 in $\str{B}$ is an edge of a copy of $\str{A}$.  It follows that vertices $f,g$ of $\str{C}$ are connected by edge of type 0 if and only if for each $1\leq l\leq N$ we have an edge of type 0 in between $f(l)$ and $g(l)$, and consequently we have $s_\str{B}(f(l))=s_\str{B}(g(l))$. Thus $\approx_\str{C}$ and $\equiv$ coincide. This proves that $\str{C}\in  \Forb^\approx_{L^\bullet}(\mathcal T^\bullet)$

For completeness we check that $\str{C}\longrightarrow (\str{B})^\str{A}_2$. Let $N$ be the Hales-Jewett number guaranteeing a monochromatic line in any $2$-colouring of $N$-dimensional cube over alphabet $\{1,2,\ldots,t\}$.

Now assume that we have a $2$-colouring of all copies of $\str{A}$ in $\str{C}$.
Using the definition of $\str{C}$ we see that  among these copies of $\str{A}$ are copies induced by an $N$-tuple    $(\widetilde{\str{A}}_{u(1)}, \widetilde{\str{A}}_{u(2)},\ldots,\widetilde{\str{A}}_{u(N)})$ of copies of $\str{A}$ for every function $u:\{1,2,\ldots, N\}\to \{1,2,\ldots, t\}$. However such copies are coded by the elements of the cube $\{1,2,\ldots,t\}^N$ and thus there is a monochromatic 
combinatorial line. This line in turn will lead to a copy $\widetilde{\str{B}}$  of $\str{B}$ in $\str{C}$ with all edges  of the form (a), (b) described above. 
\end{proof}

We can now invoke the Partite Construction \cite{Nevsetvril1989,Nevsetvril1995}
 in its standard form.
We prove:
\begin{thm}
\label{thm:mainII}
Let $\str{A},\str{B}\in \B^\bullet$  be  $L^\bullet$-struc\-tures with star equivalences $\approx_{\str{A}}$ and $\approx_{\str{B}}$ (induced by $\rel{A}{E_0}$ and $\rel{B}{E_0}$).  (Thus also $\str{A},\str{B}\in \Forb^\approx_{L^\bullet}(\mathcal T^\bullet)$ for $\mathcal T^\bullet$ defined in Theorem~\ref{thm:triangles} and Theorem~\ref{thm:triangles2}). Then there exists  $L^\bullet$-struc\-ture $\str{C}\in \Forb^\approx_{L^\bullet}(\mathcal T^\bullet)$ with the star equivalence $\approx_{\str{C}}$ induced by star forest $\rel{C}{E_0}$ such that  $$\str{C}\longrightarrow (\str{B})^{\str{A}}_2$$ with respect to embeddings preserving the equivalences.
\end{thm}

\begin{proof}
Fix structures $\str{A}, \str{B}$.  Using Theorem~\ref{mainI} obtain $\str{C}_0\in \Forb_{L^\bullet}(\mathcal T^\bullet)$ (i.e. without the star forest condition) that satisfies $\str{C}_0 \longrightarrow (\str{B})^{\str{A}}_2$. Assume without loss of generality that $C_0 = \{1,2,\ldots,c\}$.
Enumerate all copies of $\str{A}$ in $\str{C}_0$  as  $\{\widetilde{\str{A}}_1, \widetilde{\str{A}}_2,\ldots, \widetilde{\str{A}}_b\}$.
We shall define $\str{C}_0$-partite structures $\str{P}_0, \str{P}_1,\ldots, \str{P}_b$ which, as we shall show, will all belong to $\Forb^\approx_{L^\bullet}(\mathcal T^\bullet)$ with the star equivalence which induces $\approx_{\str{P}_i}$. Putting $\str{C} = \str{P}_b$ we shall have the desired Ramsey property $\str{C}\longrightarrow (\str{B})^{\str{A}}_2$.

We denote the parts of $\str{C}_0$-partite structures $\str{P}_i$ as $\mathcal X_{\str{P}_i} = \{X^{i}_1, X^{i}_2, \dots,X^{i}_c\}$.
As usual the structures $\str{P}_i$ are called {\em pictures}.
Pictures will be constructed by induction on $i$. 

The picture $\str{P}_0$ is constructed as a disjoint union of  copies of $\str{B}$: for every copy $\widetilde{\str{B}}$ of
$\str{B}$ in $\str{C}_0$ we consider a new isomorphic and disjoint copy $\widetilde{\str{B}}'$ in $\str{P}_0$ which intersects part $X^{0}_l$  if and only if the image of the projection  $\widetilde{\str{B}}$ contains vertex $l$ (so the projection restricted to $\widetilde{\str{B}}'$ is $\widetilde{\str{B}}$). Clearly $\str{P}_0\in \Forb^\approx_{L^\bullet}(\mathcal T^\bullet)$. The star forest of $\str{P}_0$ is induced by the star forest of all copies $\widetilde{\str{B}}'$.

Let the picture $\str{P}_i\in\Forb^\approx_{L^\bullet}(\mathcal T^\bullet)$ be already constructed.  Let $\widetilde{A}_i = \{x_1,x_2, \ldots,x_a\}$ be the vertices of $\widetilde{\str{A}}_i$            
(in the order of $\str{C}_0$). Let $\str{B}_i$ be the substructure of  $\str{P}_i$ induced by $\str{P}_i$ on the union of vertices of those copies of $\str{A}$ which projects to $\widetilde{\str{A}}_i$. (Note that $\str{B}_i$ need not contain all vertices of $\str{P}_i$ in parts $X^i_{x_1}, X^i_{x_2},\ldots,X^i_{x_a}$.) In this situation we use Partite Lemma~\ref{partlem} to obtain
an $\str{A}$-partite structures $\str{C}_{i+1}$  with parts  $X^{i+1}_{x_1}, X^{i+1}_{x_2},\ldots,X^{i+1}_{x_a}$.
Now consider all substructures of $\str{C}_{i+1}$ which are isomorphic to $\str{B}_i$ and extend each of these structures to a copy of $\str{P}_i$ (thus some new vertices may be added even in parts $X^i_{x_j}$). These copies are disjoint outside $\str{C}_{i+1}$, however in this extension we preserve the parts of all the copies.
The result of this multiple amalgamation of copies of $\str{P}_i$ is $\str{P}_{i+1}$. The star forest of $\str{P}_{i+1}$ is defined as an amalgamation of star forest of copies of $\str{P}_i$. Of course we have to check below that this indeed results in a star forest.

Put $\str{C} = \str{P}_b$.
It  follows easily (by now by a standard argument cf.~\cite{Nevsetvril1989,Nevsetvril1995}) that  $\str{C}\longrightarrow (\str{B})^{\str{A}}_2$:
by a backward induction one proves  that in any $2$-colouring of ${\str{C}}\choose{\str{A}}$ there exists a copy $\str{P}$ of 
$\str{P}_0$ such that the colour of a copy of $\str{A}$ in $\str{P}$ depends only on its projection.
As this in turn induces colouring of copies of $\str{A}$ in $\str{C}_0$, we obtain a monochromatic copy of 
$\str{B}$.

We have to check that the edges of $\rel{C}{E_0}$ form a star forest and that $\str{C}$ belongs to  $ \Forb^\approx_{L^\bullet}(\mathcal T^\bullet)$.
To do so we proceed by an induction on $i=0,1,2,\ldots,b$.
The statement is clear for picture $\str{P}_0$.
 In the induction step ($i \implies i+1$) we have to inspect the amalgamation of copies of $\str{P}_i$ along the copies of structures $\str{B}_i$ in $\str{C}_{i+1}$. It is clear that $\str{P}_{i+1}$ belongs again to $\Forb_{L^\bullet}(\mathcal T^\bullet)$ (as the forbidden substructures in $\mathcal T^\bullet$ are all pure-irreducible). It remains to show that $\approx_{\str{P}_{i+1}}$ is a star equivalence of $\str{P}_{i+1}$.  Because $\str{A}$ is assumed to be a good bowtie-free structure and because $\str{B}_i$ has every vertex in a copy of $\str{A}$, we know that every star 
with leaf in $\str{B}_i$ also contains its centre in $\str{B}_i$.  By Lemma~\ref{partlem} we a get a star equivalence on $\str{C}_{i+1}$. The star equivalence is preserved by the free amalgamations of $\str{P}_i$ over $\str{C}_{i+1}$ because
every time we unify leaves of a star we also unify the centre. Consequently the edges of $\nbrel{\str{P}_{i+1}}{E_0}$ form a star forest inducing in $\str{P}_{i+1}$ a star equivalence $\approx_{\str{P}_{i+1}}$.
\end{proof}
\section{Putting it together: Bowtie-free graphs have a Ramsey lift}
\label{sec:final}

In this section we prove the main Theorem~\ref{thm:intro} in the following form ($\overline{\B}$ is defined in Definition~\ref{def:L2}):
\begin{thm}
\label{thm:mainIII}
The class $\overline{\B}$ is a Ramsey class.
\end{thm}
\begin{proof}
Let $\str{A},\str{B}\in \overline{\B}$ be fixed. We prove the existence of Ramsey object $\str{C}\in\overline{\B}$
in several steps.  

Without loss of generality we can assume that $\str{B}\in \B_2$.
For $\str{A}\in \overline \B$ there exists, up to isomorphism,  a unique minimal
 $L_2$-struc\-ture $\widehat{\str{A}}\in \B_2$ which corresponds to a good bowtie-free graph such that $\str{A}$ is a substructure of  $\widehat{\str{A}}$.
This just means that we complete each centre $c(\str{A})$ to a ``full'' centre by possibly adding to every vertex left-, or  right-vertex or completing some vertices to $K_4$.  
This correspondence $\str{A}\to\widehat{\str{A}}$ is functorial in the
sense that (for any  $L_2$-structure $\str{B}\in \mathcal B_2$) the correspondence of $\str{B}\choose\str{A}$ and
$\str{B}\choose\widehat{\str{A}}$ is one to one.  (This is another consequence of the algebraic closure.) Thus we may assume that
both $\str{A}$ and $\str{B}$ are  $L_2$-struc\-tures (in the sense of Definition~\ref{def:L2}).

Let $\str{A}^\bullet, \str{B}^\bullet\in \B^\bullet$ be reduced structures of $\widehat{\str{A}}$ and $\str{B}$  with star forests inducing equivalences $\approx_{\str{A}}$ and $\approx_{\str{B}}$.
By Theorem~\ref{thm:mainII}  there exists $\str{D}^\bullet\in \Forb^\approx_{L^\bullet}(\mathcal T^\bullet)$ satisfying
$$\str{D}^\bullet\longrightarrow (\str{B}^\bullet)^{\str{A}^\bullet}_2.$$ with $\rel{{D}^\bullet}{E_0}$ forming a star forest and defining star equivalence $\approx_\str{D}$.
Without loss of generality we assume that all vertices and tuples in relations of $\str{D}^\bullet$ are contained in a copy of
$\str{B}^\bullet$. 

We use $\str{D}^\bullet$ to reconstruct
$\str{C}\in \mathcal B_2\subseteq \overline{\mathcal B}$ which is a ``completion'' of $\str{D}^\bullet$: $\str{C}$ will contain $\str{D}^\bullet$ as a
non-induced substructure in a way that every copy of $\str{B}^\bullet$ in
$\str{D}^\bullet$ can be extended to induced copy of $\str{B}$ in $\str{C}$. 

$\str{D}^\bullet$ is a reduced structure. First we reverse the reduction process. Denote by $\str{D}$ a structure
created from $\str{D}^\bullet$ by adding, for every vertex $(u)\in \rel{D}{\ell}$, a new vertex $v$, adding $(v)$ to $\rel{D}{r}$, and
connecting every vertex $u'$, $u'\approx_\str{D} u$, to $v$ by an edge in $\rel{D}{E_0}$. The result of this operation is a chimney. Similarly, for every  $(u)\in \rel{D}{1}$  introduce the additional 3 vertices
to form the clique. Finally add the edges $\rel{D}{E_1}$ connecting the newly introduced vertices
to vertices of $\str{D}^\bullet$ as described by the existing tuples in $\rel{D^\bullet}{t_i}$. Because
we assume that every vertex and every tuple of every relation of $\str{D}^\bullet$ is in a copy of $\str{B}^\bullet$
and every vertex in $\str{D}^\bullet$ is a part of a star equivalence where all copies share the centre, there is unique way of doing so:
consider $a_1\approx_\str{D} a_2\in \str{D}^\bullet$ and $b_1\approx_\str{D} b_2$ such that $(a_1,b_1)\in \rel{D^\bullet}{t_i}$ and $(a_2,b_2)\in \rel{D^\bullet}{t_j}$.
We have a copy $\widetilde{B}_1^\bullet$ containing
$\{a_1,b_1\}$ and a copy $\widetilde{B}_2^\bullet$ containing $\{a_2,b_2\}$. By lemma~\ref{lem:amalg2} we know the amalgamation is possible. It follows that types $t_i$ and $t_j$ must agree on the newly introduced edges. By iterating this we obtain structure $\str{D}$ where we also extend every isolated triangle into a chimney in the only way which is consistent with the lift.

Next we check that the $L_1$-shadow $\sh(\str{D})$ is a good ordered bowtie-free graph. 
$\sh(\str{D}^\bullet)$ is triangle-free and thus every triangle in $\str{D}$ contains at least one new central vertex. Every new triangle is a part of a copy of $\sh(\str{B})$ (which was created by expansion of a copy of $\str{B}^\bullet$) and thus it consists of edges of type 0  only. Consequently we only need to verify that edges in $\rel{D}{E_0}$ do not form a bowtie.
It is easy to verify that those edges
however forms cliques $C_4$ (which were introduced by expanding vertex in $\rel{D^\bullet}{1}$) and chimneys spanning
vertices of each star equivalence class of $\str{B}^\bullet$ which centre is in $\rel{D^\bullet}{\ell}$.

The ordering of $\str{D}^\bullet$ is not necessarily admissible. The use of
contracted structures along with the notion of admissible ordering in the sense
of Theorem~\ref{thm:NR} makes the order of $\str{D}^\bullet$ satisfy conditions
1, 2, and 3 of Definition~\ref{def:order}. The order of non-central vertices
is however free.  It is not difficult to see that the non-central vertices  can
be reordered according their centres preserving relative order of vertices
with the same centre.  This order is admissible and does preserve all
embeddings of admissibly ordered structures we need.
The order of $\str{D}$ can then be defined from the admissible order of $\str{D}^\bullet$
in a natural way.

We have constructed good ordered $L_1$-structure $\sh(\str{D})$ where every copy of shadow $\sh(\str{B}^\bullet)$
was extended to a copy of $\sh(\str{B})$.  Finally we put $\str{C}=L_2(\sh(\str{D}))\in \overline{\B}$. Because the lift $L_2$
is constructed in an unique way which preserves substructures, we have
$$\str{C}\longrightarrow (\str{B})^\str{A}_2.$$
\end{proof}

\section{The lift property}
\label{sec:expansion}
We say that a lifted class $\mathcal K^+$ has the {\em lift  property} ({\em expansion property} in~\cite{The2013}) {\em relative to $\mathcal K$}
(where $\mathcal K$ is shadow of $\mathcal K^+$) if and only if for every $\str{A}\in \mathcal K$ there is
$\str{B}\in \mathcal K$ such that for every $\str{A}^+,\str{B}^+\in \mathcal K^+$ such that
the shadow of $\str{A}^+$ is $\str{A}$ and the shadow of $\str{B}^+$ is $\str{B}$ there is an
embedding from $\str{A}^+$ to $\str{B}^+$.

The lift property is a generalisation of the {\em ordering property}~\cite{Nevsetvril1976a}. Structures
that have both Ramsey and ordering property play important role in \cite{Kechris2005} 
where they are used to obtain universal minimal flows.  \cite{The2013} defines expansion property (which for the consistency we call here the lift property) and gives results analogous to~\cite{Kechris2005} in
the setting of Ramsey classes with the lift property.  To apply these
results to the class $\overline{\B}$ we now have to show that lifts constructed in this
paper have the lift property. 

\begin{thm}
\label{thm:extendproperty}
$\overline{\B}$ has the lift property relative to $\B$.
\end{thm}
\begin{proof}
This result follows in an analogy to \cite{Nevsetvril1976a} (where it is shown that the ordered edge Ramsey property implies ordering property). Essentially we only need to deal
with additional unary and binary relations present in our lifts.

Fix a bowtie-free graph $\str{A}$.  We will give explicit construction of $\str{B}$ needed
for the lift property. Consider all possible extensions of
$\str{A}$ into a good bowtie-free graphs that are minimal in the sense that 
removing any non-empty set of vertices from the extension makes the graph either not good or
not containing $\str{A}$.  (Clearly this is a finite set.) Now consider all admissible orderings (Definition~\ref{def:order}) of these
extensions.  Denote these ordered good bowtie-free graphs by
$\str{A}_1,\str{A}_2,\ldots, \str{A}_N$.

Now we extend graphs $\str{A}_1,\str{A}_2,\ldots,
\str{A}_N$ to graphs  $\str{A}'_1,\str{A}'_2,\ldots, \str{A}'_N$ by simple gadgets that will allow us to use the Ramsey property to ensure the order.
For that we consider all vertices $v$ of $\str{A}_i$ ($1\leq i\leq N)$ with the following properties:
\begin{enumerate}
\item $M_I=\{v;(v)\in \nbrel{L_2(\str{A}_i)}{\ell}\}$,
\item $M_{II}=\{v;(v)\in \nbrel{L_2(\str{A}_i)}{1}\}$,
\item $M_{III}=\{v;(v)$ is in no unary relations of $L_2(\str{A}_i)\}$.
\end{enumerate}
Sets $M_I$, $M_{II}$ and $M_{III}$ are chosen in a way that the orders of sets $M_I$, $M_{II}$ and $M_{III}$ together with admissibility of order (Definition~\ref{def:order}) determine total ordering of  all vertices.

Consider pair of vertices $u\neq v$ of $\str{A}_i$ where both $u$ and $v$ belong to one of the sets $M_{I}$, $M_{II}$ or where both $u$ and $v$ belong to set $M_{III}$ and have the same centre.
For each such pair extend ordered good bowtie-free graph $\str{A}_i$ in a way so there is a vertex $w(u,v)$ that belongs to same class in the lift as $u$ and $v$, it is not connected
by an edge to $u$ nor $v$, and it is in between $u$ and $v$ in the order of $\leq_{\str{A}^+_i}$.  Such a vertex can always be added in a way that the result is an ordered good bowtie-free graph (by possibly
introducing new chimney or a copy of $K_4$). We denote by $\str{A}'_i$ an ordered good bowtie-free graph having vertex $w(u,v)$ for
every possible choice of $u$ and $v$ in $\str{A}_i$.

Denote by $\str{A}'$  the disjoint union of graphs $\str{A}'_1,\str{A}'_2,\ldots, \str{A}'_N$.
Now for every pair of $u<_{\str{A}}w$ for which we introduced $w(u,v)$ consider substructures
induced on $\{u,w(u,v)\}$ and $\{w(u,v), v\}$ by $L_2(\str{A}')$.

These structures
are always isomorphic. (Recall that $\rel{}{\ell},\rel{}{r},\rel{}{1},\rel{}{2},\rel{}{3},\rel{}{4}$ are the only unary  relations in $L_1$.)
Denote isomorphism types of those structures by $\str{E}_I$, $\str{E}_{II}$ and $\str{E}_{III}$. Now we use
Theorem~\ref{thm:mainII} to get:
\begin{eqnarray*}
 \str{C}_{I}&\longrightarrow&(L_2(\str{A}'))^{\str{E}_I}_2,\\
 \str{C}_{II}&\longrightarrow&(\str{C}_{I})^{\str{E}_{II}}_2,\\
 \str{C}_{III}&\longrightarrow&(\str{C}_{II})^{\str{E}_{III}}_2,
\end{eqnarray*}
where all $\str{C}_{I}, \str{C}_{II}, \str{C}_{III}$ belong to $\overline{\B}$ .
Let $\str{B} \in \B_0$  be the shadow of $\str{C}_{III}$. We claim that $\str{B}$ is a good bowtie-free graph with the lift property for $\str{A}$.

Let $\str{B}^+\in \overline{\B}$ be a $L_2$-lift of any admissibly ordered structure $\str{B}$.
We now assign colours to copies of $\str{E}_{I},\str{E}_{II},$ and $
\str{E}_{III}$ in $\str{C}_{III}$ by comparing order of vertices in $\str{B}^+$ and
$\str{C}_{III}$. (If the order agrees the colour is red and blue otherwise.)  By Ramsey property we obtain a copy of $\str{A}'$ such all
copies of $\str{E}_I$, $\str{E}_{II}$, $\str{E}_{III}$ are monochromatic.  This
means that within the copies of $\str{A}_i$ the relative order of vertices
within sets $M_I$, $M_{II}$ and set $M_{III}$ vertices assigned to a given centre is
either the same as in $\str{B}^+$ or opposite (independently in each class).
It is easy to see that any admissible orderings of $\str{A}$ can be turned to
another admissible ordering of $\str{A}$ by reversing orders within each of
the classes (and possibly adjusting order in between intervals assigned to each
centre). It is thus possible to find $\str{A}_i$ within $\str{B}^+$ that is
ordered the same way.
\end{proof}
In the language of \cite{Kechris2005,The2013} we thus obtain the following corollary.
\begin{corollary}
The automorphism group of the \Fraisse{} limit of $\overline{\B}$ is extremely amenable
and this expansion gives universal minimal flow of the automorphism group of the \Fraisse{} limit of $\B$.
\end{corollary}

\begin{remark}
Theorem~\ref{thm:extendproperty} is the only place in the
paper that actually needs conditions given on admissible ordering by
Definition~\ref{def:order}. There are many possible choices of orderings of
good bowtie-free graphs; completely free ordering, ordering by unary relations, ordering by
corresponding centre, etc.  Good graphs ordered freely (as well as other cases)
can also be shown to have Ramsey lift constructed the same way as our lift.
It is easy to see that such class
however fails to have the lift property:  Consider a good bowtie-free graph $\str{A}$
consisting of two chimneys where order of non-central vertices does not follow
the order of centres. For any choice of $\str{B}$ it is possible to give an
admissible ordering in the sense of Definition~\ref{def:order} giving a lift of
$\str{B}$ that does not include the given lift of $\str{A}$ (because the
order of $\str{A}$ is not admissible).

Theorem~\ref{thm:extendproperty} can thus be understood as an argument why
among possible Ramsey lifts of $\B$ the one given here is the optimal one.
\end{remark}

\section{Concluding remarks}
\label{remarks}
The existence of universal objects is a difficult question in its own  and this is also where the class of all bowtie-free graphs played a vital role. 
 Here is a brief history:
It starts with a (somewhat surprising)  result of Komj\' ath \cite{Komjath1999} that the class of bowtie-free graphs $\mathcal B$ contains a {\em universal graph}, i.e. there exists a (countably infinite) bowtie-free graph $U$  such that every finite or countably infinite  bowtie free graph $G$  has an embedding into $U$ (in other words, $G$ is isomorphic to an {\em induced} subgraph of $U$).
In a sense this obscurely looking example was (and is) a key case for further development (see e.g. \cite{Ackerman2012,Cherlin1999,Cherlin2007,Cherlin2011,Cherlin2015,Cherlin2001,Cherlinb}). 

Note that  the problem of characterising universal graphs seem to be far from being solved even in the following special case:  Given a finite set of finite graphs $\F$, denote by $\Forb_M(\F)$ the class of all finite or countable infinite graphs which do not contain any $F\in \F$ as a (not necessarily induced) subgraph. For which $\F$ does the class $\Forb_M(\F)$ have a universal graph? (Non-induced subgraphs correspond to monomorphism and $M$ in $\Forb_M(\F)$ stands for monomorphism.)
The answer is positive for the bowtie graph while, for example,  the answer is negative for the rectangle $C_4$. It is not even known whether this question, for a general finite $\F$, is decidable~\cite{Cherlin2011}.

The problem was recast in the model theory setting by Cherlin et al.~\cite{Cherlin1999}. They narrowed the search for universal graphs to more structured ultrahomogeneous and $\omega$-categorical graphs (and structures) and in
\cite{Cherlin1999} they provided a structural characterisation of such universal structures:
There is an $\omega$-categorical universal graph in $\Forb_M(\F)$ if and only if the class of existentially complete graphs in $\Forb_M(\F)$ has a locally finite algebraic closure. Bowtie-free graphs fall in this category.

Let us formulate in this setting a consequence of our construction of $\overline{\B}$ in Section~\ref{sec:homogenization}.
Observe that in the amalgamations involved in the proof of Lemma~\ref{lem:amalgamation} the non-central vertices of $\str{B}_1$ and $\str{B}_2$ are identified if and only if they belong to $\str{A}$ (so amalgamation is ``strong on non-central vertices'').
Consequently by the standard \Fraisse{} argument we get:

\begin{corollary}
\label{cor:ultrahomogeneous}
The class of all finite structures in $\overline{\B}$ is the age of an ultrahomogeneous structure and its shadow is a universal graph for class $\B$.
\end{corollary}

\begin{remark}
This, of course, follows also from Proposition~1 in~\cite{Cherlin1999} where  the existence of $\omega$-categorical universal object is established. However, here  we provided an explicit construction by means of finite lifts. In fact this is the first such explicit lift (compare~\cite{Cherlin1999}) and this is of independent interest~\cite{Hubicka2013,Hubicka2009}.
\end{remark}

It is conjectured in \cite{Cherlin2011} that for
classes defined by forbidden monomorphisms from one forbidden graph the algebraic
closure operator is either unary or there is no universal $\omega$-categorical graph at
all. We believe that all such classes with unary closure operator can be proved to be Ramsey
by a generalisation of a proof presented here.
On the other hand a simple example is given in  \cite{Cherlin2011} showing that the closure 
does not need to be unary for classes defined by forbidden homomorphism from
more than one connected graph. Our techniques does not seem to directly generalise for
this case.

Other important case is the situation where the amalgamation is not free over
closed sets. Several such classes with strong amalgamation have been proved to be
Ramsey by means of Partite Construction (among those the classes mentioned in the
introduction: partial orders, metric spaces and
classes $\Forb_H(\F)$).

We hope that it is possible to combine both techniques to obtain Ramsey results on even more restricted
classes of graphs.

\medskip

\noindent {\em Acknowledgement.} We thank to a referee for remarks which improved quality of the presentation.

\bibliography{emanuel.bib}

\end{document}